\documentclass{amsart}
\usepackage{amsmath}
\usepackage{amsthm}
\usepackage{amssymb}
\usepackage{amsaddr}
\usepackage{mathrsfs}
\usepackage{enumerate}
\usepackage{graphicx}
\usepackage{fullpage}

\newcommand{\sL}{\mathcal L}

\newcommand{\p}[1]{\left( {#1} \right)}

\newcommand{\sC}{\mathcal C}

\newtheorem{theorem}{Theorem}
\newtheorem{lemma}{Lemma}

\newtheorem{remark}{Remark}

\title[Bochner Integral Form of Operator-Valued Riccati Equations]{Well-Posedness of the Bochner Integral Form of Operator-Valued Riccati Equations}
\author{James Cheung}
\address{Millennium Space Systems, A Boeing Company. 2265 E. El Segundo Blvd, El Segundo, CA, 90245.}
\begin{document}

\maketitle

\begin{abstract}
    In this short paper, we prove that the Bochner integral form of the operator-valued Riccati equation has a unique solution if and only if its mild form has a unique solution. This implies that the mild and Bochner integral forms of this equation are equivalent. The result is obtained through an operator representation argument.
\end{abstract}

\section{Introduction}

Let $H$ be a separable Hilbert space equipped with the inner product $(\cdot, \cdot)_H$. We define $\sL(H)$ to be the space of bounded linear operators defined on $H$. We will then denote $A: \mathcal D(A) \rightarrow H$ as the generator of a $C_0$-semigroup $S(t) \in \sL(H)$ for all $t \in [0, \tau]$, where $\tau > 0$ and $\mathcal D(A)$ is the domain of $A$ defined densely in $H$. The solution space of interest in this work is $\sC([0,\tau], \sL(H))$, which defines the space of bounded operators that are norm-continuous with respect to $t \in [0,\tau]$.

In the differential form, the operator-valued Riccati equation is given by 
\begin{equation} \label{eqn: DARE}
\left\{
\begin{aligned}
    \frac{d}{dt} \Sigma(t) &= A\Sigma(t) + \Sigma(t)A^* + \Sigma(t) G \Sigma(t) - F \\
    \Sigma(0) &= \Sigma_0,
\end{aligned}
\right. 
\end{equation}
for all $t \in [0,\tau]$, where  $F, \Sigma_0 \in \sL(H)$ are self-adjoint operators, and $G$ is an unbounded self-adjoint operator whose domain is dense in $H$. The mild form of this equation is then given by
\begin{equation} \label{eqn: Mild Form}
    \Sigma(t)\phi = S(t) \Sigma_0 S^*(t)\phi + \int_0^t S(t-s)\p{F - \Sigma G \Sigma}S^*(t-s)\phi ds
\end{equation}
for all $\phi \in H$ and $t \in [0,\tau]$. From the results presented in \cite{bensoussan2007representation}, we know that it is generally known that \eqref{eqn: DARE} and \eqref{eqn: Mild Form} are equivalent, meaning that there exists a unique $\Sigma(\cdot) \in \sC([0,\tau], \sL(H))$ that satisfies both equations. In this paper, we demonstrate that $\Sigma(\cdot) \in \sC([0,\tau], \sL(H))$ satisfying \eqref{eqn: Mild Form} also satisfies
\begin{equation} \label{eqn: Bochner Form}
    \Sigma(t) = S(t) \Sigma_0 S^*(t) + \int_0^t S(t-s)\p{F - \Sigma G \Sigma}S^*(t-s) ds
\end{equation}
for all $t\in[0,\tau]$.

The well-posedness of the Bochner integral form of the operator-valued Riccati equation plays an important part in determining theoretical error bounds for approximations to this equation \cite{burns2022optimal, cheung2023approximation}. The previously known result presented in \cite{burns2015solutions} indicates that the Bochner integral form of the operator-valued Riccati equation is well-posed if the operators $F,G$ in \eqref{eqn: Bochner Form} are compact. This result was derived through an approximation argument. This work extends well-posedness to cases where $G$ is not necessarily bounded. We proceed to prove this result in the following section.

\section{Analysis}
In the analysis, we will use an operator representation argument to demonstrate that the mild form and the Bochner integral form of the operator-valued Riccati equation are equivalent. To this end, we will utilize the following corollary to the Riesz Representation Theorem found in \cite[Theorem A.63]{hall2013quantum}.

\begin{lemma} \label{lemma: quadratic form}
    If $q(\cdot): H \rightarrow \mathbb R$ is a bounded quadratic form on $H$, then there exists a unique self-adjoint operator $Q \in \sL(H)$ such that
    $$
        q(\phi) = \p{\phi, Q \phi}_H
    $$
    for all $\phi \in H$.
\end{lemma}

We now move to prove the main result of this work given in the following.

\begin{theorem}
Let $S(t) \in \sL(H)$ be a $C_0$-semigroup defined on $t \in [0,\tau]$. Now, suppose that there exists an unique time-dependent self-adjoint operator $\Sigma(\cdot) \in \sC([0, \tau], \sL(H))$ that satisfies the following mild form of the operator-valued Riccati equation
\begin{equation} \label{eqn: mild form}
    \Sigma(t)\phi = S(t) \Sigma_0 S^*(t)\phi + \int_0^t S(t-s)\p{F - \Sigma G \Sigma}S^*(t-s)\phi ds
\end{equation}
for all $\phi \in H$ and $t\in [0,\tau]$, where $F, \Sigma_0 \in \sL(H)$ are bounded self-adjoint operators and $G$ is a generally unbounded self-adjoint operator whose domains is dense in $H$. Then $\Sigma(\cdot) \in \sC([0,\tau], \sL(H))$ satisfies \eqref{eqn: mild form} if and only if it satisfies also the following Bochner integral form of the operator-valued Riccati equation
\begin{equation} \label{eqn: bochner form}
    \Sigma(t) = S(t) \Sigma_0 S^*(t) + \int_0^t S(t-s)\p{F - \Sigma G \Sigma}S^*(t-s) ds
\end{equation}
for all $t \in [0,\tau]$.
\end{theorem}
\begin{proof}
    Since $\Sigma(\cdot) \in \sC([0,\tau], \sL(H))$ satisfies \eqref{eqn: mild form}, then we must have that
    \begin{equation} \label{eqn: quadratic form}
        \p{\phi, \Sigma(t) \phi}_H = 
            \p{\phi, S(t)\Sigma_0 S^*(t)\phi}_H +
            \p{\phi, \int_0^t S(t-s)\p{F - \Sigma G \Sigma}(s) S^*(t-s)\phi ds}_H
    \end{equation}
    for all $\phi \in H$ and $t \in [0, \tau]$. 
    
    Defining 
    \begin{equation*}
    \begin{aligned}
        q^1_t(\phi) &:= \p{\phi, \Sigma(t)\phi}_H \\
        q^2_t(\phi) &:= \p{\phi, S(t)\Sigma_0 S^*(t) \phi}_H \\
        q^3_t(\phi) &:= \p{\phi, \int_0^t S(t-s)\p{F-\Sigma G \Sigma}(s)S^*(t-s)\phi ds}_H
    \end{aligned}
    \end{equation*}
    as quadratic forms defined for all $\phi \in H$ and $t\in [0,\tau]$. The boundedness of $q^1_t(\cdot)$ follows from the observation that $\Sigma(t)\phi \in H$ for all $\phi \in H$ and $t \in [0,\tau]$. Equation \eqref{eqn: mild form} then requires that $S(t) \Sigma_0 S^*(t)\phi \in H$ and that $\int_0^t S(t-s)\p{F-\Sigma G \Sigma}(s) S^*(t-s)\phi ds \in H$ for all $\phi \in H$ and $t \in [0,\tau]$, which implies the boundedness of $q^2_t(\cdot), q^3_t(\cdot)$. Applying Lemma \ref{lemma: quadratic form} then implies that there exists unique operators $Q^1_t, Q^2_t, Q^3_t \in \sL(H)$ so that
    \begin{equation*}
    \begin{aligned}
        q^1_t(\phi) = \p{\phi, Q^1_t \phi}_H \\
        q^2_t(\phi) = \p{\phi, Q^2_t \phi}_H \\
        q^3_t(\phi) = \p{\phi, Q^3_t \phi}_H
    \end{aligned}
    \end{equation*}
    for all $\phi \in H$ and $t\in [0,\tau]$.
    
    It then follows from \eqref{eqn: quadratic form} that 
    \begin{equation*}
        q^1_t(\phi) = q^2_t(\phi) + q^3_t(\phi)
    \end{equation*}
    for all $\phi \in H$ and $t \in [0,\tau]$. This can only be true if
    \begin{equation*}
        Q^1_t = Q^2_t + Q^3_t
    \end{equation*}
    for all $t \in [0,\tau]$. Then, by the definition of $q^1_t, q^2_t, q^3_t$ and the uniqueness of $Q^1_t, Q^2_t, Q^3_t$ (implied by Lemma \ref{lemma: quadratic form}) associated with their respective quadratic forms, we have necessarily that
    \begin{equation*}
    \begin{aligned}
        Q^1_t &= \Sigma(t) \\
        Q^2_t &= S(t)\Sigma_0 S^*(t) \\
        Q^3_t &= \int_0^t S(t-s)\p{F - \Sigma G \Sigma}(s) S^*(t-s)ds,
    \end{aligned}
    \end{equation*}
    for all $t \in [0,\tau]$. Hence, $\Sigma(\cdot) \in \sC\p{[0,\tau], \sL(H)}$ must also satisfy 
    \begin{equation*}
        \Sigma(t) = S(t) \Sigma_0 S^*(t) + \int_0^t S(t-s)\p{F - \Sigma G \Sigma}(s) S^*(t-s)ds
    \end{equation*}
    for all $t \in [0,\tau]$. Thus we have proven the ``if'' part of the theorem. The proof in the other direction follows by testing \eqref{eqn: bochner form} with any $\phi \in H$.
\end{proof}

\begin{remark}
We would like to point out that the analysis presented in the proof above implies that the operator-valued integral
\begin{equation*}
    \int_0^t S(t-s)\p{F - \Sigma G \Sigma}(s) S^*(t-s)ds 
\end{equation*}
is the unique representation of the bounded self-adjoint time-dependent linear operator $Q_t\in \sC([0,\tau],\sL(H))$ that satisfies 
\begin{equation*}
    \p{\phi, Q_t \phi}_H := \p{\phi, \int_0^t S(t-s)\p{F-\Sigma G \Sigma}(s)S^*(t-s)\phi ds}_H
\end{equation*}
for all $\phi \in H$ and $t \in [0,\tau]$. This indicates that the operator-valued integral used in the Bochner integral form of the operator-valued Riccati equation is well-defined. 
\end{remark}

\section{Discussion}
We have demonstrated above that the mild and Bochner integral forms of the operator-valued Riccati equation are equivalent. Instead of using an approximation argument, as done in \cite{burns2015solutions}, we have utilized an operator representation argument to achieve this result. This simpler proof then allows us to extend the known well-posedness results for the Bochner integral form to cases where the coefficient operator $G$ in the equation are unbounded. 

In following works, the author will utilize the result presented in this paper to determine error bounds for approximation methods to operator-valued Riccati equations for cases where the coefficient operators $G$ in the equation are defined by boundary and point control/observation  operators.

\bibliographystyle{plain}
\bibliography{bibliography}

\end{document}